\documentclass[11pt,a4paper]{article}

\usepackage{amsfonts}
\usepackage{amsmath,amssymb}
\usepackage{cases}
\usepackage{stmaryrd}
\usepackage{epsfig}
\usepackage{pstricks}
\usepackage{graphicx}
\usepackage{enumerate}
\usepackage{url}
\usepackage{shorttoc}
\usepackage[breaklinks,%
					linktocpage,%
					colorlinks,%
					backref=page,%
					pdftitle={convection-diffusion}]{hyperref}
\hypersetup{colorlinks=true,citecolor=magenta,urlcolor=blue}
\usepackage[hyphenbreaks]{breakurl}
\usepackage[nottoc,notlot,notlof]{tocbibind}
\usepackage{hypernat}
\usepackage{algorithm}
\usepackage{algorithmicx,algpseudocode}

\usepackage{calc}
\usepackage{longtable}
\usepackage{multirow}

\usepackage{framed}
\usepackage{lscape}

\usepackage{fullpage}

\graphicspath{{}}

\newtheorem{theo}{\bf Theorem}[section]
\newtheorem{lem}[theo]{\bf Lemma}
\newtheorem{pro}[theo]{\bf Proposition}

\newenvironment{Proofc}[1]{\smallskip\par\noindent\textbf{#1}\quad}%
  {\hfill$\Box$\bigskip\par}
\newenvironment{proof}{\begin{Proofc}{Proof}}{\end{Proofc}}

\newcommand{\R}{{\mathbb R}}

\usepackage{amsopn}

\DeclareMathOperator*{\dx}{\mbox{\small\textsf{dx}}}
\DeclareMathOperator*{\dt}{\mbox{\small\textsf{dt}}}
\DeclareMathOperator*{\dxb}{\mbox{\footnotesize\textsf{dx}}}
\DeclareMathOperator*{\dtb}{\mbox{\footnotesize\textsf{dt}}}
\DeclareMathOperator*{\fle}{\rightarrow}

\newcommand{\per}{\varepsilon}

\newcommand{\eof}{{\unskip\nobreak\hfil\penalty50
          \hskip2em\hbox{}\nobreak\hfil\mbox{\rule{1ex}{1ex} \qquad}
   \parfillskip=0pt
   \finalhyphendemerits=0\par\medskip}} 

\newcommand{\ba}{\begin{array}{ll}}
\newcommand{\ea}{\end{array}}

\newcommand{\be}{\begin{equation}}
\newcommand{\ee}{\end{equation}}

\newcommand{\bi}{\begin{itemize}}
\newcommand{\ei}{\end{itemize}}

\newcommand{\bc}{\begin{center}}
\newcommand{\ec}{\end{center}}

\newcommand{\bfig}{\begin{figure}[!ht]}
\newcommand{\efig}{\end{figure}}

\newcommand{\ben}{\begin{enumerate}}
\newcommand{\een}{\end{enumerate}}

\newcommand{\bmat}{\left[\begin{matrix}}
\newcommand{\emat}{\end{matrix}\right]}

\begin{document}

\title{From traffic and pedestrian follow-the-leader models with reaction time 
to first order convection-diffusion flow models} 

\author{
Antoine Tordeux\thanks{Forschungszentrum J\"ulich GmbH and Bergische Universit\"at Wuppertal, Germany
    (\href{mailto:a.tordeux@fz-juelich.de}{a.tordeux@fz-juelich.de}).}
  \and
  Guillaume Costeseque\thanks{INRIA Sophia Antipolis--M\'editerran\'ee, France (\href{mailto:guillaume.costeseque@inria.fr}{guillaume.costeseque@inria.fr}).}
  \and
  Michael Herty\thanks{Rheinisch--Westfälische Technische Hochschule Aachen, Germany
    (\href{mailto:herty@igpm.rwth-aachen.de}{herty@igpm.rwth-aachen.de}).}
		\and
  Armin Seyfried\thanks{Forschungszentrum J\"ulich GmbH and Bergische Universit\"at Wuppertal, Germany
    (\href{mailto:a.seyfried@fz-juelich.de}{a.seyfried@fz-juelich.de}).}
}

\maketitle

\begin{abstract}
In this work, we derive first order continuum traffic flow models from a microscopic delayed follow-the-leader model. 
Those are applicable in the context of vehicular traffic flow as well as pedestrian traffic flow. 
The microscopic model is based on an {\em optimal} velocity function and a reaction time parameter.  
The corresponding macroscopic formulations in Eulerian or Lagrangian coordinates 
result in first order convection-diffusion equations. 
More precisely, the convection is described by the optimal velocity while the diffusion term depends on the reaction time. 
A  linear stability analysis for homogeneous solutions of both continuous and discrete models are provided. 
The conditions match the ones of the car-following model for specific values of the space discretization.
The behavior of the novel model is illustrated thanks to numerical simulations. 
Transitions to collision-free self-sustained stop-and-go dynamics are obtained if the reaction time is sufficiently large. 
The results show that the dynamics of the microscopic model can be well captured by the macroscopic equations. 
For non--zero reaction times we observe a scattered fundamental diagram. The scattering width is compared
to real pedestrian and road traffic data. 
\end{abstract}


\textbf{AMS:}
  35F20, 70F45, 90B20, 65M12.

\section{Introduction} 

Microscopic and macroscopic approaches for the purpose of vehicular traffic flow modelling
have been often developed separately in the engineering community  
\cite{btre,Helbing1997aa,BayenClaudel2010aa,Lebacque1993aa}. 
Similar models can also be used in the description of
pedestrian dynamics \cite{SchadschneiderSeyfried2011aa,ChraibiKemlohSchadschneider2011aa,Appert-RollandDegondMotsch2011aa}. 
Typically, microscopic models are based on the so-called ``follow-the-leader'' strategy and 
they are stated as (finite of infinite) systems of Ordinary Differential Equations (ODEs).  
They are generally based on speed or acceleration functions which depend on distance spacing, speed, predecessor' speed, relative speed and so on. 
One of the simplest approach is a speed model solely based on the spacing, firstly proposed by Pipes \cite{Pipes1953}
\be
\dot x_i(t)=W(\Delta x_i(t)),
\label{pip}
\ee
where $\Delta x_i(t)=x_{i+1}(t)-x_i(t)$ denotes the spacing between the vehicle $(i)$ to its predecessor $(i+1)$ 
and $W(\cdot)$ stands for the equilibrium (or optimal) speed function depending on the spacing.
The microscopic models are discrete in the sense that the vehicles or pedestrians $i \in \mathbb{Z}$ 
are individually considered. 
A macroscopic description consider the flow of vehicles or pedestrians (in the following also referred to as \emph{agents}) 
as a continuum in Eulerian or Lagrangian coordinates. 
For instance in the most classical Eulerian time-space framework, 
the main variables are the density, the flow and the mean speed. 
The simplest approach is the scalar hyperbolic equation of the celebrated 
Lighthill-Whitham-Richards (LWR) model \cite{lw,r}
\be
\partial_t\rho+\partial_x(\rho V(\rho))=0. 
\label{lwr}
\ee
Here $\rho$ is the density, $V(\cdot)$ is the equilibrium speed function which is assumed to depend only on the density. 
The flow $f(\rho)=\rho V(\rho)$  is given by the product of the density times the mean speed. 
The model is derived from the continuity equation for which the flow is supposed in equilibrium. 
The microscopic and macroscopic models Eq.~(\ref{pip}) and Eq.~(\ref{lwr}) well reproduce shock-wave phenomena for Riemann problems. 
Yet such models are not able to describe the observed transition to scattered flow/density relation 
(the fundamental diagram) with hysteresis and self-sustained stop-and-go phenomena (see \cite{Treiterer1974,Kerner1997,Chowdhury2000} and Fig.~\ref{fd-data}).  
This is due to the fact that spatially homogeneous regime are always in the equilibrium solutions and determined  
by the functions $W(\cdot)$ and $V(\cdot)$, respectively.

\bfig\bc
\includegraphics[width=\textwidth]{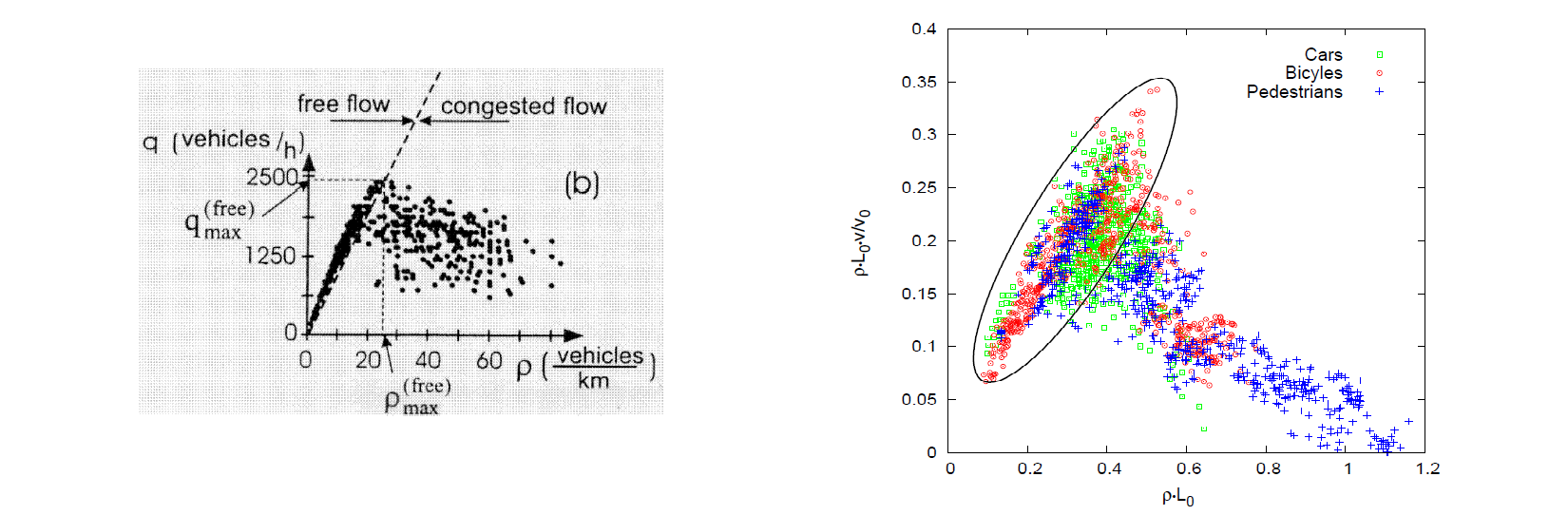}\vspace{-2mm}
\caption{Empirical fundamental diagrams. Left, \cite[Figure 1]{Kerner2000} and right, \cite[Figure 5]{ZhangJ2014}.}
\label{fd-data}
\ec\efig

Therefore, the microscopic behavior is modified by introducing reaction and relaxation times. 
The simplest following model of this type may be the delayed model by Newell \cite{Newell1961}
\be
\dot x_i(t+\tau)=W(\Delta x_i(t)),
\label{new}
\ee
with $\tau$ the reaction time (if positive). 
Applying a Taylor expansion in the l.h.s.~of the delayed speed model Eq.~(\ref{new}), 
we obtain the second order `optimal velocity model' (OVM) introduced by Bando \textit{et al.} in \cite{Bando1995}. 
The OVM has 
limit-cycles in stationary states, with self-sustained propagation of non-linear stop-and-go waves, and 
hysteresis curves in the fundamental flow/density diagram (see \cite{Orosz2006,Orosz2010}). 
Macroscopic second-order models comprised of systems of hyperbolic equations are also able to reproduce 
non-linear stop-and-go waves and scattering of the fundamental diagram. 
One of the first approach is the one by Payne and Whitham \cite{Payne1971,Whitham2011}. 
The model can be derived from the microscopic Newell model Eq.~(\ref{new}). 
The main drawback of this model is that, as pointed out by Daganzo \cite{Daganzo1995}, 
the speed and the density could yield negative values and are not bounded. 
Note that this drawback is also observed with follow--the--leader models like the OVM and is referred as 
collision between the vehicles (see for instance \cite{Davis2003,Orosz2009} or \cite[Chap.~15]{btre}). 
Aw and Rascle have corrected this issue by replacing the space derivative 
of the `pressure' by a convective derivative \cite{Aw2000} (AR model). 
Nowadays extensions of the AR model such as the ARZ, GARZ or generalized models 
\cite{Greenberg2002,Zhang2002,Berthelin2008,Flynn2009,Seibold2013,Fan2014}, 
as well as two phase models coupled  with the LWR model \cite{Colombo2003,Goatin2006,Colombo2010,Blandin2011},
are used to describe transition to congested traffic with scattered fundamental diagrams and self-sustained non-linear shock waves. 
A general framework is the generic second order model (GSOM) family introduced in \cite{Lebacque2007,Lebacque2005}. 
Most of the approaches are a posteriori based on the continuous description. 

In this article, we derive minimalist macroscopic traffic flow models of first order
from a microscopic speed model to describe stop-and-go wave phenomena and scattering of the fundamental diagram.
The use of first order models allow us to ensure by construction that 
the speed and the density remain positive and bounded.  
The starting point is a OV microscopic model of first order including a reaction time parameter. 
We show in Sec.~\ref{models} that the corresponding macroscopic model results in a convection-diffusion equation. 
The macroscopic model is discretized using distinct Godunov and Euler-based schemes and 
the linear stability conditions for the homogeneous solutions of these numerical schemes are provided in Sec.~\ref{num}. 
The conditions match the ones of the car-following model for specific values of the spatial discretization step.
Simulations are carried out in Sec.~\ref{sim}. 
Systems with different initial conditions are numerically solved. 
Further, we compare with data of realistic traffic flow as well as pedestrian flow. 

\section{Microscopic and macroscopic models} \label{models}

\subsection{The microscopic follow-the-leader model}

The microscopic model we use has been introduced in \cite{Tordeux2014}. 
It is based on the Newell model Eq.~(\ref{new}).
In the remaining of the paper, we assume that $W : s \mapsto W(s)$ is Lipschitz continuous, 
non-decreasing and upper bounded
in order to get the well-posedness of Eq.~(\ref{new}) 
supplemented with initial conditions $x_i (t=0) = x_{i,0}$ for any $i \in \mathbb{Z}$.
We rewrite the equation as 
\be
 \dot x_i(t) = W( x_{i+1}(t-\tau) - x_i(t-\tau)),
\ee
and apply a Taylor expansion in the argument of $W$. 
Neglecting higher--order terms in $\tau$ we obtain 
\be
\dot x_i(t)= W\big(\Delta x_i(t)-\tau\big[W(\Delta x_{i+1}(t))-W(\Delta x_i(t))\big]\big).
\label{mmi}
\ee
The model is a system of ordinary differential equations of first order with two predecessors in interaction. 
It is calibrated by the delayed time $\tau\in\mathbb R$, that is a reaction time if positive and an anticipation time if negative, 
and the optimal speed function $W(\cdot)$. 
The function $W(\cdot)$ is supposed to be bounded by $V_0>0$, positive and zero if the spacing is smaller than $\ell$, 
$\ell>0$ being the vehicle's length or size of the pedestrian.  
Note that the model admits a minimum principle, say $\Delta x_i(t)\ge\ell$ $\forall i,t$.
Thus it is by construction collision-free and it has the same 
stability condition as the initial microscopic Newell model Eq.~(\ref{new}) or as the OVM from \cite{Bando1995}. 
This condition is for all $s \in \mathbb{R}$
\be
|\tau| W'(s) <1/2.
\label{csmicro}\ee
Note that the condition simply reduces to $|\tau|< T/2$ if one considers the linear fundamental diagram 
$W : s \mapsto W(s) := \max \left\{ 0 , \min \left\{ \frac{1}{T} (s-\ell) , V_0 \right\} \right\}$, with $T > 0$.
When unstable, the model transits to states with collision-free self-sustained
stop-and-go dynamics, see \cite{Tordeux2014}. 

\subsection{Derivation of macroscopic models}

In the following, we consider $i=1,\ldots,N$ agents with periodic boundary conditions 
(i.e.~the predecessor of the agent $N$ is the agent $1$). 
The derivation of macroscopic models from microscopic models is useful to fully understand the dynamics.   
In \cite{Aw2002}, Aw \textit{et al.} established the connection between a microscopic car-following model 
and the second-order AR macroscopic traffic flow model. 
The rigorous proof, based on a scaling limit where the time and space linearly increase 
while the speed and the density remain constant, 
assumes homogeneous conditions. 
We use here the same methodology considering the local density $\rho_i(t)$ around the vehicle or pedestrian $(i)$ 
and at time $t > 0$,  
as the inverse of the spacing  
\be
\rho_i(t) :=\frac{1}{\Delta x_i(t)}.
\label{dr}
\ee
The density could also be normalized by multiplication with $\ell.$ Here, we prefer to keep the unit of one over length as density 
to ease the comparison with the classical models.
Then, the microscopic model reads 
\be
\dot x_i(t)= W \left( \frac{1}{\rho_i(t)} -\tau \left[ W \left(\frac{1}{\rho_{i+1}(t)}\right) 
- W\left(\frac{1}{\rho_i(t)}\right) \right] \right)
=:\tilde V(\rho_{i+1}(t),\rho_i(t)),
\ee
for a velocity profile $\tilde{V}.$ Then, 
\be\label{Vtilde}
\partial_t \frac{1}{\rho_i(t)}= \partial_t \Delta x_i(t) = 
\tilde V(\rho_{i+2}(t),\rho_{i+1}(t))-\tilde V(\rho_{i+1}(t),\rho_i(t)).
\ee
In \cite{Aw2002} it has been observed that Eq.~(\ref{Vtilde}) is a semi--discretized version of hyperbolic partial differential equation
in Lagrangian coordinates. This requires to consider limits of many vehicles or pedestrians 
$N\to \infty$ and diminishing length $\ell \to 0$. 
We introduce the continuous variable $y \in \mathbb{R}$ such that $y_i=i \Delta y$ as counting variable for the number of agents 
where $\Delta y$ is proportional to $\ell$. 
By piecewise constant extension of the given spacing, we construct a density 
$\rho(t,y)$ such that 
$\displaystyle \frac{1}{\rho_i(t)} = \frac{1}{\Delta y} \int_{y_i-\frac{\Delta y}2}^{y_i+\frac{\Delta y}2} \frac{1}{\rho(t,z)} dz$.  
The quantity $\frac{1}{\rho_i(t)}$ is the volume average over a cell of length $\Delta y$ centered at $y_i$. 
The r.h.s. of \eqref{Vtilde} describes the flux across the cell boundaries. 
Introduce $V : k \mapsto V(k)=W(\frac{1}{k})$ for any $k>0$. Then, it follows that
$\tilde{V}(k_1,k_2) = V \left( \frac{k_2}{1 - k_2 \tau \left[ V(k_1) - V(k_2) \right]} \right)$
for any $(k_1,k_2) \in (0,+\infty)^2$ satisfying $V(k_1) \neq V(k_2) + \frac{1}{\tau k_2}$.
As for OVM, $W$ is non-decreasing, therefore we observe that $V$ is non-increasing on $(0,+\infty)$. 
We obtain from \eqref{Vtilde}
\be
\partial_t \frac{1}{ \rho_i(t)} - \frac{\Delta y}{ \Delta y } 
 \left[ V \left( \frac{\rho_{i+1}}{ 1- \tau \rho_{i+1} Z_{i+1}  } \right)  
 - V \left( \frac{\rho_{i}}{ 1- \tau \rho_{i} Z_i  } \right) \right] = 0,
\ee
 where $Z_i:=V( \rho_{i+1}) - V(\rho_i)$. Provided that $-V$ is increasing and $\rho(t,y)$ is piecewise constant on each cell, 
the reconstruction at
 the cell interface $y_{i \pm \frac{1}2}$ is given by cell averages, i.e., 
 $\frac{1}{\rho \left( t,y_{i+\frac{1}2} \right) }=\frac{1}{\rho_{i+1}(t)}.$ 
 Next, 
 we rescale time $t \to t   \Delta y$ and also reaction time $\tau \to \tau \Delta y$ to obtain 
\be\label{prefinal}
\partial_t \frac{1}{ \rho_i(t)} - \frac{1}{ \Delta y } 
 \left[  V \left( \frac{\rho_{i+1}}{ 1-  \tau \rho_{i+1} \frac{Z_{i+1}}{\Delta y} } \right)  -
 V \left( \frac{\rho_{i}}{ 1- \tau \rho_{i} \frac{Z_{i}}{\Delta y} } \right) \right] = 0.
\ee
 Hence, we observe that in the rescaled time and in the limit $\Delta y \to 0$ the microscopic model is an upwind discretization of the following macroscopic equation
 \be 
\partial_t\frac1\rho -\partial_y V \left( \frac{\rho}{1-\tau\rho\partial_y V(\rho)} \right)=0.
\label{mma1L}
\ee
The upwind or Godunov scheme is the most mathematically reasonable discretization
 provided $\tau$ is sufficiently small due to the decreasing behavior of $V$ for suitable OVM functions $W.$ 
Up to second--order in $\tau$ we approximate \eqref{mma1L} by Taylor expansion and obtain a convection--diffusion model as 
\be
\partial_t\frac1\rho-\partial_yV(\rho)=\tau\partial_y\big((\rho V'(\rho))^2\partial_y\rho\big).
\label{mma2L}
\ee

The relation between the density in Lagrangian coordinates and Eulerian coordinates is given by the coordinate transformation $(t,y)\to(t,x)$ 
where  $y=\int^x_{-\infty} \rho(t,x) dx.$ Note that $y$ counts the number of vehicles/pedestrians up to position $x$ in Eulerian coordinates.
In the Eulerian coordinates $(t,x)$, the macroscopic model Eq.~(\ref{mma1L}) reads
\be
\partial_t\rho + \partial_x\left(  
\rho V\left(  \frac{ \rho}{ 1-\tau \partial_x V(\rho) }
    \right)\right)=0.
		\label{mma1}
\ee
The model could be seen as an extension of the LWR model Eq.~(\ref{lwr}) 
with a modified speed-density relationship  
$\rho\mapsto V\left( \rho/(1-\tau \partial_x V(\rho)) \right)$.  
For illustrating the behavior of this modified speed-density mapping, we set $I := \tau \partial_x V(\rho))$
and we define $\mathcal V:(\rho,I)\mapsto V\left(\rho/(1-I)\right)$.
The fundamental diagrams obtained for a constant term $I \in \left\{-0.3, 0, 0.3 \right\}$ 
and for a speed function $V:\rho\mapsto\max\{0,\min\{2,1/\rho-1\}\}$ are shown on Figure~\ref{fdex}.
Note that for constant (in space) densities $\rho$ (and/or for $\tau = 0$), the additional
term $I$ vanishes and we recover the classical LWR model.

\bfig\bc\includegraphics[width=.9\textwidth]{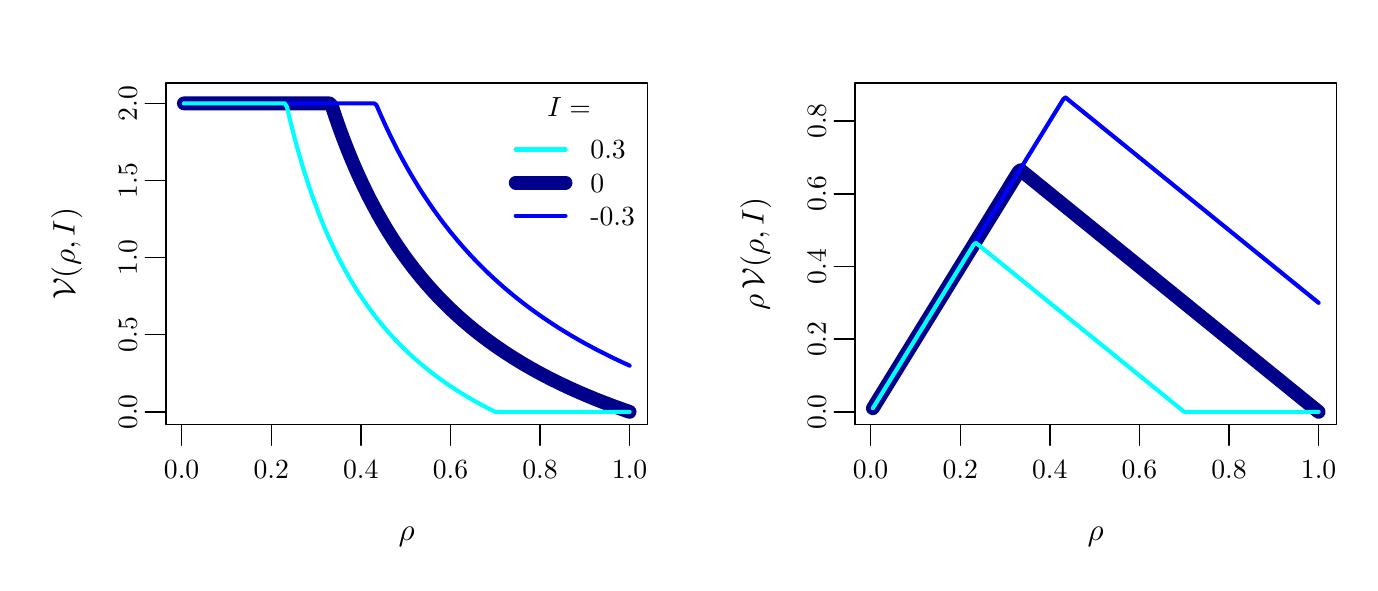}\vspace{-5mm}
\caption{Illustration for the fundamental diagram $\mathcal V:(\rho,I)\mapsto V\big(\rho/(1-I)\big)$ 
obtained in the macroscopic model~\eqref{mma1} with constant 
inhomogeneity $I\in \left\{-0.3, 0, 0.3 \right\}$ and $V:\rho\mapsto\max\{0,\min\{2,1/\rho-1\}\}$.}
\label{fdex}\ec\efig

\section{Linear stability analysis}\label{num}

\subsection{Linear stability analysis of the continuous macroscopic model}

A Taylor expansion up to second-order in terms of $\tau$ for equation~(\ref{mma1}) yields
\be
\partial_t \rho + \partial_x (\rho V(\rho))
  = -\tau \partial_x \left((\rho V'(\rho))^2 \partial_x \rho \right).
\label{mma2}
\ee
We also consider an initial condition
\be
\rho(t=0,x) = \rho_0(x) \quad \mbox{for any} \quad x \in \mathbb{R}
 \label{mma2-ini}
\ee
where $\rho_0 \in L^1(\mathbb{R}) \cap \text{BV}(\R)$.
By defining $\displaystyle D(\rho) := \int_{-\infty}^{x} -\tau (\rho V'(\rho))^2 \partial_x \rho \ \text{d}y$,
we obtain an equation similar to the one considered in~\cite{Burger2003}, say
\be\partial_t \rho + \partial_x (\rho V(\rho))
  = \partial^2_x D(\rho).\ee
One can verify that $D(\rho_0)$ is absolutely continuous on $\mathbb{R}$
and that $\partial_x D(\rho_0) \in \text{BV}(\mathbb{R})$.
In \eqref{mma2}, the l.h.s.~is the LWR model with additional diffusion 
proportional to the reaction time parameter $\tau$
and that can be either negative or positive. 
More precisely the diffusion is negative in deceleration phases where the density get higher upstream, and it is positive 
in the opposite acceleration phases. 
This type of diffusion seems to induce an instability of the homogeneous (constant) solutions and the formation of oscillations (i.e.~jam waves).  
The diffusion coefficient $(\rho V'(\rho))^2$ depends on the density and the fundamental diagram. 
In fluid dynamics the coefficient is a characteristic for the flexibility of the random movement responsible for the diffusion. 
In traffic flows, comparable diffusion-convection forms have been used in \cite{Nelson2000,Burger2003}. 
We refer the interested reader to \cite{Burger2003} (and references therein) for a proof 
of existence and uniqueness of the solution to Eq.~(\ref{mma2})-\eqref{mma2-ini}.
In the following we analyze the linear stability of homogeneous solutions for the macroscopic model.

\begin{pro}\label{stabcont}
The homogeneous configurations for which $\rho(x,t)=\rho_e$ for all $x$ and $t$ are linearly stable for the continuous 
traffic model Eq.~(\ref{mma2}) if and only if 
\be
\tau<0.
\ee
Note that a negative $\tau$ refers to an anticipation time. 
\end{pro}

\begin{proof}
If $\per(x,t)=\rho(x,t)-\rho_e$ is a perturbation to homogeneous solution $\rho_e$, then 
\be
\per_t=F(\rho_e+\per,\per_x,\per_{xx})=\alpha\per+\beta\per_x+\gamma\per_{xx}+o(LC(\per,\per_x,\per_{xx})),
\label{modp}
\ee
with $F(\rho,\rho_x,\rho_{xx})=-\partial_x \big(\rho V(\rho)-\tau(\rho V'(\rho))^2\partial_x\rho\big)$, 
$\alpha=\frac{\partial F}{\partial \rho}(\rho_e,\rho_e,\rho_e)=0$, 
$\beta=\frac{\partial F}{\partial \rho_x}=-V(\rho_e)-\rho_eV'(\rho_e)$,  
$\gamma=\frac{\partial F}{\partial \rho_{xx}}=-\tau(\rho_eV'(\rho_e))^2$. 
The solutions of the linear system are the Ansatz $\per=z e^{\lambda t-ixl}$ where 
$\lambda\in\mathbb C$, $x,l\in\mathbb R$. 
We get $\per_t=\lambda\per$, $\per_x=-il\per$ and $\per_{xx}=-l^2\per$. 
Therefore the characteristic equation of the perturbed system Eq.~(\ref{modp}) is 
$\lambda_l=\tau (l\rho_eV'(\rho_e))^2+il(V(\rho_e)+\rho_eV'(\rho_e))$. 
The homogeneous solution are stable when $\varepsilon\fle0$, i.e.~$\Re(\lambda_l)<0$ for all $l>0$. 
This holds only if the diffusion is positive. This is $\tau<0$. 
\eof 

Therefore the macroscopic model is unstable as soon as the reaction time $\tau$ is positive which is 
the physically reasonable case. An explanation is that the Taylor expansion of the original model
in terms of $\tau$ does lead to  a perturbed equation with different properties. However,
the discrete model does not have this stability requirement. Therefore, we show below that 
for suitable discretization of the model we recover stability. 

\end{proof}

\subsection{Linear stability analysis for the discrete schemes}

Discretizations of the macroscopic models Eqs.~(\ref{mma1L}) and (\ref{mma2L}) in `Lagrangian' coordinates 
give the initial microscopic model Eq.~(\ref{mmi}).
Our purpose in this section is the discretization of the macroscopic models Eqs.~(\ref{mma1}) and (\ref{mma2}) 
in Eulerian coordinates. 
We denote $\dt$ and $\dx$ the time and spatial discretization steps 
and use the Godunov scheme \cite{Godunov1959aa} for the discretization of the density
\be
\rho_i(t+\dt)=\rho_i(t)+\frac{\dt}{\dx}\big(f_{i-1}(t)-f_i(t)\big)
\label{d0}
\ee
where $f_i$ denotes the flow at cell boundary and has to be determined.
For this aim, we introduce the \emph{demand} and \emph{supply} functions
from the flow-density fundamental diagram $f : \rho \mapsto f(\rho) := \rho V(\rho)$ 
as first proposed in \cite{Daganzo1995,Lebacque1996} and that read respectively
\be\Delta (\rho) := \max_{k \leq \rho} f(k) \quad \mbox{and} \quad \Sigma(\rho) := \max_{k \geq \rho} f(k)\ee
and we define the Godunov flux as $G(x,y) :=\min\{\Delta(x),\Sigma(y)\}$.
We are now ready to propose three different strategies to compute the boundary flows $f_i$. 
The first two methods discretize the linearized model Eq.~(\ref{mma2}) using a splitting scheme 
which treats separately the convection and the diffusion terms. 
The last scheme is a simple discretization of the exact macroscopic model Eq.~(\ref{mma1}).
\ben
\item \textbf{The Godunov/Euler scheme:} a Godunov scheme for the convection term and an explicit Euler scheme for the diffusive term 
of the linearised model Eq.~(\ref{mma2}):
\be
f_i^{(1)}=G(\rho_{i},\rho_{i+1})+\frac{\tau}{\dx}(\rho_i V'(\rho_i))^2(\rho_{i+1}-\rho_i).
\label{d1}
\ee
Such a scheme is the one used in \cite{Aw2002}.

\item \textbf{The Godunov/Godunov scheme:} a Godunov scheme for the convection term and a Godunov scheme for the diffusion term 
of the Taylor--expanded model Eq.~(\ref{mma2}):
\be
f_i^{(2)}=G(\rho_{i},\rho_{i+1})+
\frac{\tau}{\dx}\rho_i V'(\rho_i)\big[G(\rho_{i+1},\rho_{i+2})-G(\rho_{i},\rho_{i+1})\big].
\label{d2}
\ee

\item \textbf{The Godunov scheme:} a Godunov scheme for the modified convection term in the exact macroscopic model Eq.~(\ref{mma1}):
\be
f_i^{(3)}=G\left(\frac{\rho_{i}}{1-\frac{\tau}{\dxb}\big(V(\rho_{i+1})-V(\rho_{i})\big)},
\frac{\rho_{i+1}}{1-\frac{\tau}{\dxb}\big(V(\rho_{i+2})-V(\rho_{i+1})\big)}\right).
\label{d3}\ee
Note that this scheme is valid if $1-\frac{\tau}{\dxb}\big(V(\rho_{i+1})-V(\rho_{i})\big)>0$ for all $\rho_i$ and $\rho_{i+1}$. 
By denoting $V_0=\sup_x V(x)$, this inequality holds if 
\be\tau <\dx/V_0.
\label{CFL}
\ee 
\een


\begin{pro}\label{stabdisc}
The homogeneous configurations for which $\rho_i(t)=\rho_e$ for all $i$ and $t$ are linearly stable for the 
discrete traffic model Eq.~\ref{d0} if and only if 
\be 
\alpha^2+\beta^2+\gamma^2+\xi^2-2\alpha\gamma-2\beta\xi+2f(c_l)<1,\qquad \forall l=1,\ldots,N-1,
\label{CG}
\ee 
with $c_l=\cos(2\pi l/N)$ and $f(x)=(\alpha\beta+\alpha\xi+\beta\xi-3\gamma\xi) x+2(\alpha\gamma+\beta\xi) x^2+4\gamma\xi x^3$, 
and $\alpha=\frac{\partial F}{\partial \rho_i}$, $\beta=\frac{\partial F}{\partial \rho_{i+1}}$, 
$\gamma=\frac{\partial F}{\partial \rho_{i+2}}$ and $\xi=\frac{\partial F}{\partial \rho_{i-1}}$ 
the partial derivatives of the model in equilibrium.
\end{pro}
\begin{proof}
The perturbations to homogeneous solution are the variables $\per_i(t)=\rho_i(t)-\rho_e$. 
The perturbed system is
\be\begin{array}{lcl}
\per_i(t+\dt)&=&\rho_i(t+\dt)-\rho_e=
F(\rho_i(t),\rho_{i+1}(t),\rho_{i+2}(t),\rho_{i-1}(t))-\rho_e\\[2mm]
&=&\alpha\,\per_i(t)+\beta\,\per_{i+1}(t)+\gamma\,\per_{i+2}(t)+\xi\,\per_{i-1}(t)+o(\mbox{LC}(\per_{i},\per_{i-1},\per_{i+1},\per_{i+2})),
\end{array}\ee
with $\alpha=\frac{\partial F}{\partial \rho_i}$, $\beta=\frac{\partial F}{\partial \rho_{i+1}}$, 
$\gamma=\frac{\partial F}{\partial \rho_{i+2}}$ and $\xi=\frac{\partial F}{\partial \rho_{i-1}}$ at
$(\rho_e,\rho_e,\rho_e,\rho_e)$. 
General conditions for the global stability of the discrete schemes can be obtained for a system 
of $N$ cells with periodic boundary conditions.  
The linear perturbed system is 
$\vec\per\,(t+\dt)=M\,\vec\per\,(t)$,
with $\vec\per=\,^{\mbox{\tiny T}}(\per_1,\ldots,\per_N)$ and 
$M$ a sparse matrix with $(\xi,\alpha,\beta,\gamma)$ on the diagonal. 
If $M=PDP^{-1}$ with $D$ a diagonal matrix, then $\vec\per\,(t)=PD^{t/\dtb}P^{-1}\,\vec\per\,(0)\fle\vec0$ 
if all the coefficients of $D$ are less than one excepted one equal to 1. 
$M$ is circulant therefore the eigenvectors of $M$ are $z(\iota^0,\iota^1,\ldots,\iota^{m-1})$ 
with $\iota=\exp\left(i\frac{2\pi l}{N}\right)$ and $z\in\mathbb Z$, and the eigenvalues are 
$\lambda_l=\alpha+\beta\iota_l+\gamma\iota_l^2+\xi\iota_l^{-1}$.  
The system is linearly stable if $|\lambda_l|<1$ for all $l=1,\ldots,N-1$. 
This is 
\be \lambda_l^2=\alpha^2+\beta^2+\gamma^2+\xi^2-2\alpha\gamma-2\beta\xi+2f(c_l)<1,\qquad \forall l=1,\ldots,N-1,
\label{CG2}
\ee 
with $c_l=\cos(2\pi l/N)$ and $f(x)=(\alpha\beta+\alpha\xi+\beta\xi-3\gamma\xi) x+2(\alpha\gamma+\beta\xi) x^2+4\gamma\xi x^3$.
\eof

These conditions Eq.~(\ref{CG}) are applied to the different numerical schemes Eq.~(\ref{d1}), Eq.~(\ref{d2}) and Eq.~(\ref{d3})  
with optimal speed $V(\rho)=\frac1T(1/\rho-\ell)$, $T>0$ being the vehicles time gap and $\ell\ge0$ the vehicle' size. 
For such speed function, the Godunov scheme is simply $G(x,y)=\frac1T(1-y\ell)$. 

\end{proof}


\begin{lem}\label{schemeD1}
The homogeneous configurations are linearly stable for the 
Godunov-Euler scheme Eqs.~(\ref{d0}-\ref{d1})  if 
\be
2\tau<T\ell\,\mbox{\emph{$\dx$}}\,\rho_e^2,
\label{CSd1}
\ee
and if \mbox{\emph{$\dt$}} is sufficiently small.
\end{lem}
\begin{proof}
Mixed with the scheme for the density Eq.~(\ref{d0}), Godunov/Euler scheme Eq.~(\ref{d1}) is  
\be
F_1(\rho_{i},\rho_{i+1},\rho_{i+2},\rho_{i-1})=
\rho_i+\frac{\dt}{T\dx}\left(\ell(\rho_{i+1}-\rho_i)+\frac{\tau}{T\dx}\left(\frac{\rho_{i}-\rho_{i-1}}{\rho_{i-1}^2}
-\frac{\rho_{i+1}-\rho_{i}}{\rho_{i}^2}\right)\right),
\label{F1}
\ee
where $\alpha=1-A+2B$, $\beta=A-B$, $\gamma=0$ and $\xi=-B$ 
with $A=\frac{\dtb \ell}{T\dxb}$ and $B=\frac{\dtb\tau}{(T\dxb\rho_e)^2}$.

If $\tau<\frac12T\ell\dx\rho_e^2$, then $\alpha>0$ for
\be 
\dt<\frac{T\dx}{\ell+\frac{2\tau}{T\dxb\rho_e^2}},
\label{Pal2}\ee
and for all $\dt\ge0$ if $\tau\ge\frac12T\ell\dx\rho_e^2$. 
Moreover $1-\alpha>0$ if $\tau<\frac12T\ell\dx\rho_e^2$ while 
$\beta$ is positive only if $\tau<T\ell\dx \rho_e^2$ 
and the sign of $\xi$ is the one of $-\tau$.

The stability conditions are distinguished according to the sign of $\tau$.
\bi
\item If $\tau<0$ and Eq.~(\ref{Pal2}) holds, 
then $f(x)=\alpha(1-\alpha)x+2\beta\xi x^2$ is strictly convex and is maximal on $[-1,1]$ for $x=-1$ or $x=1$. 
Therefore the model is stable if $f(-1)<f(1)$; this is simply $-\alpha(1-\alpha)<\alpha(1-\alpha)$ 
that is always true since $\alpha>0$ if (\ref{Pal2}) holds and $1-\alpha>0$ on $\tau<0$.
Therefore the system is stable for all $\tau<0$.

\item Several cases have to be distinguished for $\tau>0$. 
We assume in the following that Eq.~(\ref{Pal2}) holds. 
\bi
\item For $0<\tau<\frac12T\ell\dx\rho_e^2$, we have $\alpha,1-\alpha,\beta>0$, $\xi<0$ and
$f(x)=\alpha(1-\alpha)x+2\beta\xi x^2$ is strictly concave and maximal for $x_0=-\frac{\alpha(1-\alpha)}{4\beta\xi}>0$.   
The model is stable if $x_0>1$, this is 
\be\dt<\frac{T\dx}{\ell}-\frac{2\tau}{(\ell\rho_e)^2}.
\ee
This condition is more restrictive than Eq.~(\ref{Pal2}). 
\item For $\frac12T\ell\dx\rho_e^2<\tau<T\ell\dx\rho_e^2$, we have $\alpha,\beta>0$, $1-\alpha,\xi<0$ and 
$f(-1)>f(1)$ therefore the model is unstable. 
More precisely, $f$ maximal for $x_0<-1$, i.e. the unstable solution have shortest wavelength if
\be
\dt<\left(\tau-\frac12T\ell\dx\rho_e^2\right)\left(\frac{2\tau}{T\dx\rho_e}\right)^{-2}.
\ee
This condition is also more restrictive than Eq.~(\ref{Pal2}). 
\item For $\tau>T\ell\dx\rho_e^2$, we have $\alpha>0$, $1-\alpha,\beta,\xi<0$ and the system is 
unstable for all $\dt$ with shortest wavelength since $f(\cdot)$ strictly convex and $f(-1)>f(1)$.\eof
\ei\ei

The stability conditions for the Godunov/Euler splitting scheme Eq.~(\ref{d1}) are summarised in Fig.~\ref{sum}.
The same conditions as the continuous macroscopic model are obtained for 
$\dx\fle0$ and $\dt\fle0$ such that $\dt/\dx\fle0$.

\bfig\bc\vspace{5mm}
\includegraphics[width=\textwidth]{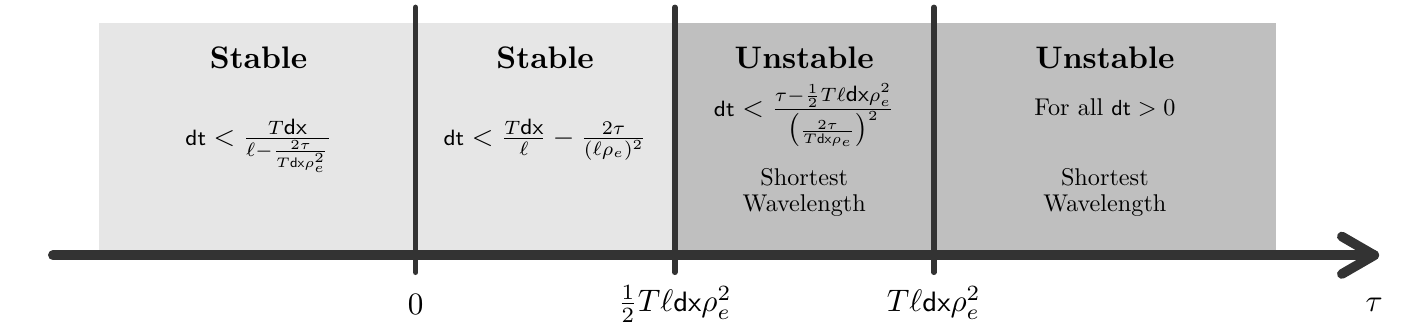}\vspace{-5mm}
\caption{Summary of the stability conditions for the Godunov/Euler splitting scheme Eq.~(\ref{d1}).}
\label{sum}\ec\efig

\end{proof}


\begin{lem}\label{schemeD2-3}
The homogeneous configurations are linearly stable for the 
Godunov-Euler schemes Eqs.~(\ref{d0}-\ref{d2}) and Eqs.~(\ref{d0}-\ref{d3})  if
\be
2|\tau|<T\,\mbox{\emph{$\dx$}}\,\rho_e,
\label{CSd2-3}
\ee
and if \mbox{\emph{$\dt$}} is sufficiently small.
\end{lem}
\begin{proof}
The Godunov numerical schemes Eq.~(\ref{d2}) and Eq.~(\ref{d3}) are respectively
\be
F_2(\rho_{i},\rho_{i+1},\rho_{i+2},\rho_{i-1})=
\rho_i+\frac{\dt \ell}{T\dx}\left(\rho_{i+1}-\rho_i+\frac{\tau}{T\dx}
\left(\frac{\rho_{i+1}-\rho_{i+2}}{\rho_i}-\frac{\rho_i-\rho_{i+1}}{\rho_{i-1}}\right)\right),
\label{F2}\ee
and 
\be
F_3(\rho_{i},\rho_{i+1},\rho_{i+2},\rho_{i-1})=
\rho_i+\frac{\dt \ell}{T\dx}\Big(
\frac{\rho_{i+1}}{1-\frac{\tau}{T\dxb}\big(\frac1{\rho_{i+2}}-\frac1{\rho_{i+2}}\big)}
-\frac{\rho_{i}}{1-\frac{\tau}{T\dxb}\big(\frac1{\rho_{i+1}}-\frac1{\rho_{i}}\big)}\Big).
\label{F3}\ee
By construction, both give $\alpha=1-A(1+B)$, $\beta=A(1+2B)$, $\gamma=-AB$ and $\xi=0$ 
with $A=\frac{\dtb \ell}{T\dxb}$ and $B=\frac{\tau}{T\dxb\rho_e}$. 
As expected, the stability conditions of these two schemes are the same. 

If $\tau>-T\dx\rho_e$, then $\alpha>0$ for
\be 
\dt<\frac{T\dx}{\ell+\frac{\ell\tau}{T\dxb\rho_e}},
\label{Pal}\ee
and for all $\dt\ge0$ if $\tau\le-T\dx\rho_e$. 
$\beta$ is positive only if $\tau>-\frac12T\dx\rho_e$.
Moreover $1-\beta>0$ if Eq.~(\ref{Pal}) holds while 
the sign of $\gamma$ is the one of $-\tau$. 

Here again, the stability conditions are distinguished according to the sign of $\tau$.
\bi
\item If $\tau<0$ and Eq.~(\ref{Pal}) holds, $f(x)=\beta(1-\beta)x+2\alpha\gamma x^2$ is strictly convex is maximal on $[-1,1]$ 
for $x=-1$ or $x=1$. Therefore the model is stable if $f(-1)<f(1)$; this is
\be
\tau>-\frac12T\dx\rho_e\qquad\mbox{and}\qquad\dt<\frac{T\dx}{\ell+\frac{2\ell\tau}{T\dxb\rho_e}}.
\ee
The condition for $\dt$ is weaker than (\ref{Pal}) since $\tau$ is negative. 
If $\tau\le-\frac12T\dx\rho_e$ then the system is unstable at the shortest wave-length frequency. 
A sufficiently condition for that the finite system produces the shortest frequency is simply $N\ge2$. 
Note that no condition holds on $\dt$ if $\tau\le-T\dx\rho_e$.


\item If $\tau>0$ and Eq.~(\ref{Pal}) holds then $f(x)=\beta(1-\beta)x+2\alpha\gamma x^2$ is concave and is maximum at
$\mbox{arg}\sup_xf(x)=x_0=-\frac{\beta(1-\beta)}{4\alpha\gamma}>0$. 
We know that $\lambda_0^2=\alpha^2+\beta^2+\gamma^2-2\alpha\beta+f(1)=1$ (case $l=0$). 
Therefore the model is stable if $x_0>1$; this is
\be
\tau<\frac12T\dx\rho_e\qquad\mbox{and}\qquad\dt<\frac{T\dx}{\ell}-\frac{2\tau}{\ell\rho_e}.
\ee
The condition for $\dt$ is stronger than Eq.~(\ref{Pal}). 
If $\tau\ge\frac12T\dx\rho_e$ then the system is unstable at the frequency $\cos^{-1}(x_0)$ that is reachable in the finite system if 
$N>2\pi/\cos^{-1}(x_0)$. 
We have $x_0\fle1/2+T\dx\rho_e/(4\tau)$ as $\dt\fle0$, going from $1$ to $1/2$ according to $\tau$ (long wave).\eof
\ei

The stability conditions for the Godunov/Godunov and Godunov schemes Eq.~(\ref{d2}) and Eq.~(\ref{d3}) are summarised in Fig.~(\ref{summary}).
The same conditions as the microscopic model are obtained at the limit $\dt\fle0$ for 
$\dx=1/\rho_e$, i.e.~a space step equal to the mean spacing.

\bfig\bc\vspace{5mm}
\includegraphics[width=\textwidth]{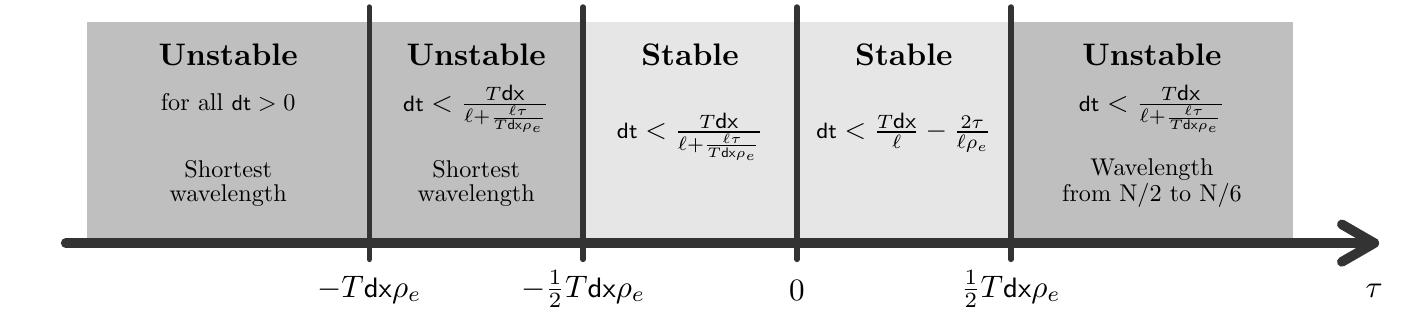}\vspace{-5mm}
\caption{Summary of the stability conditions for the Godunov/Godunov and Godunov schemes Eq.~(\ref{d2}) and Eq.~(\ref{d3}). 
Note that we have the additional condition $\tau<\dxb/V_0$, with $V_0=\sup_xV(x)$, for the simple Godunov scheme Eq.~(\ref{d3}).
}
\label{summary}\vspace{0mm}\ec\efig

\end{proof}

\subsection{Bounds on the speed and the density}

The microscopic model Eq.~(\ref{mmi}) is collision-free: 
the spacing remains by construction bigger than the vehicle length $\ell>0$, 
and the speed is positive and bounded. 
We check whether this property also occurs with the numerical schemes $F_j$, $j=1,2,3$ of the macroscopic models 
Eqs.~(\ref{F1}), (\ref{F2}) and (\ref{F3}). 
The models are from the first order therefore the speed is necessary positive and bounded if the optimal velocity functions are so defined. 
Moreover, the density remains bounded in $[1,\rho_M]$, with $\rho_M=1/\ell$, if 
\be 
F_j(0,a,b,c)\ge0\quad\mbox{and}\quad F_j(\rho_M,a,b,c)\le\rho_M\qquad\mbox{for all } (a,b,c). 
\ee
It is easy to check that such property holds only if $\tau\le0$ for the Godunov/Euler Eq.~(\ref{F1}), 
while it holds for $\tau \ge -\dx\rho_e/W'$ 
with the Godunov/Godunov scheme Eq.~(\ref{F2}), and for $\tau <\dx/V_0$ with the simple Godunov scheme Eq.~(\ref{F3}). 
As the microscopic model, Eq.~(\ref{F2}) and Eq.~(\ref{F3}) 
are able to describe macroscopically unstable homogeneous solutions with large waves by ensuring 
that speed and density remain positive and bounded. 
The relation between instability and self-sustained traffic waves (or jamiton) are notably described in 
\cite{Flynn2009,Orosz2006,Orosz2010,Seibold2013} with microscopic and macroscopic second order models. 
In the next section, we analyse by simulation the unstable solutions we get with the first order models   
for different initial conditions.

\section{Simulation results} \label{sim}

In this section numerical simulations of the microscopic model Eq.~(\ref{mmi}) 
and of the simple Godunov scheme Eq.~(\ref{F3}) macroscopic model are compared.  
The car-following model Eq.~(\ref{mmi}) is simulated using an explicit Euler scheme. 
A ring (periodic boundaries) with a length 101 and 50 vehicles is considered. 
The optimal speed functions are $W(\Delta)=\max\{0,\min\{2,\Delta-1\}\}$ and $V(\rho)=W(1/\rho)$ 
corresponding to a triangular fundamental diagram, while the reaction time is $\tau=1$.  
The values of the parameters are set to obtain unstable homogeneous solutions. 
The time step is $\dt=0.01$.  
The space step for the Godunov scheme is the mean spacing $\dx=101/50=2.02$ 
in order to match the stability conditions of both microscopic and macroscopic model (see Eq.~(\ref{csmicro}) and Fig.~\ref{summary}) 
and to hold the CFL conditions (see Eq.~(\ref{CFL}) and Fig.~\ref{summary}).   
Three experiments are carried out with different initial conditions. 
In the first one, the initial configuration is a jam. 
The initial condition is random in the second  experiment 
while it is a perturbed homogeneous configuration in the last one. 

\subsection{Trajectories}
In Figs.~\ref{spti3}, \ref{spti1} and \ref{spti2}, the  trajectories of the microscopic model 
and the time series for the density by cell for the discrete macroscopic model (gray levels) are plot for respectively 
the jam, random and perturbed initial conditions. 
The jam stationary propagates within the first experiment in Fig.~\ref{spti3}. 
Both microscopic and macroscopic models rigorously describe the same dynamics. 
The dynamics obtained does not perfectly coincide for the random and perturbed initial conditions 
(see Figs.~\ref{spti1} and \ref{spti3}). 
Yet most of the dynamics seems to be well recaptured and notably the self-sustained emergence of traffic stop-and-go waves. 
Note that the waves propagate backward with the speed $-\ell/T$ that is close to the value empirically observed (see \cite{Nishinari2003}).

\bfig\bc\vspace{0mm}
\includegraphics[width=.81\textwidth]{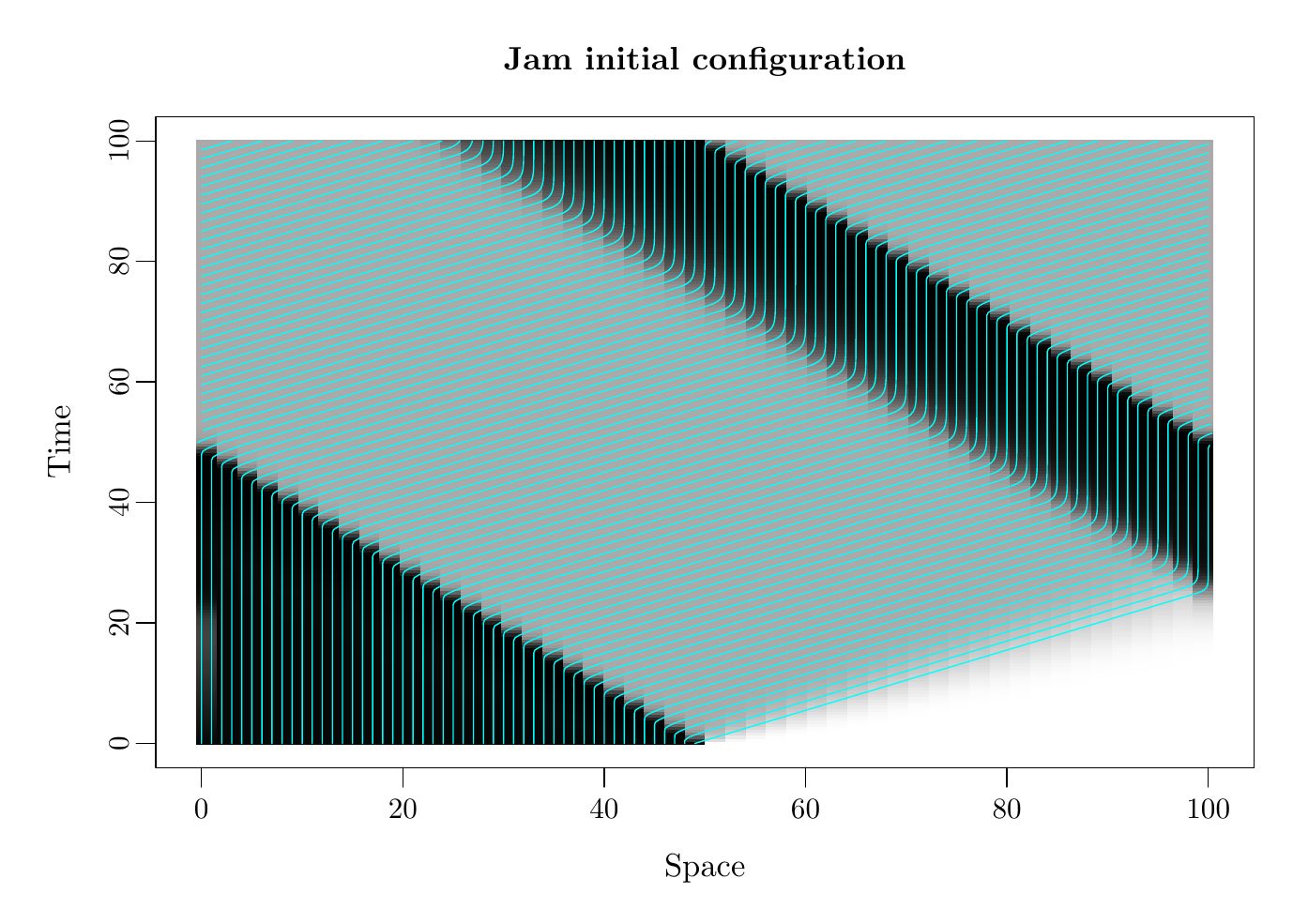}\vspace{-4mm}
\caption{The trajectories of the microscopic model 
(cyan curves) and the time series for the density by cell for the
discrete macroscopic model (gray levels) for jam initial conditions.}\label{spti3}
\ec\efig

\bfig\bc\vspace{0mm}
\includegraphics[width=.81\textwidth]{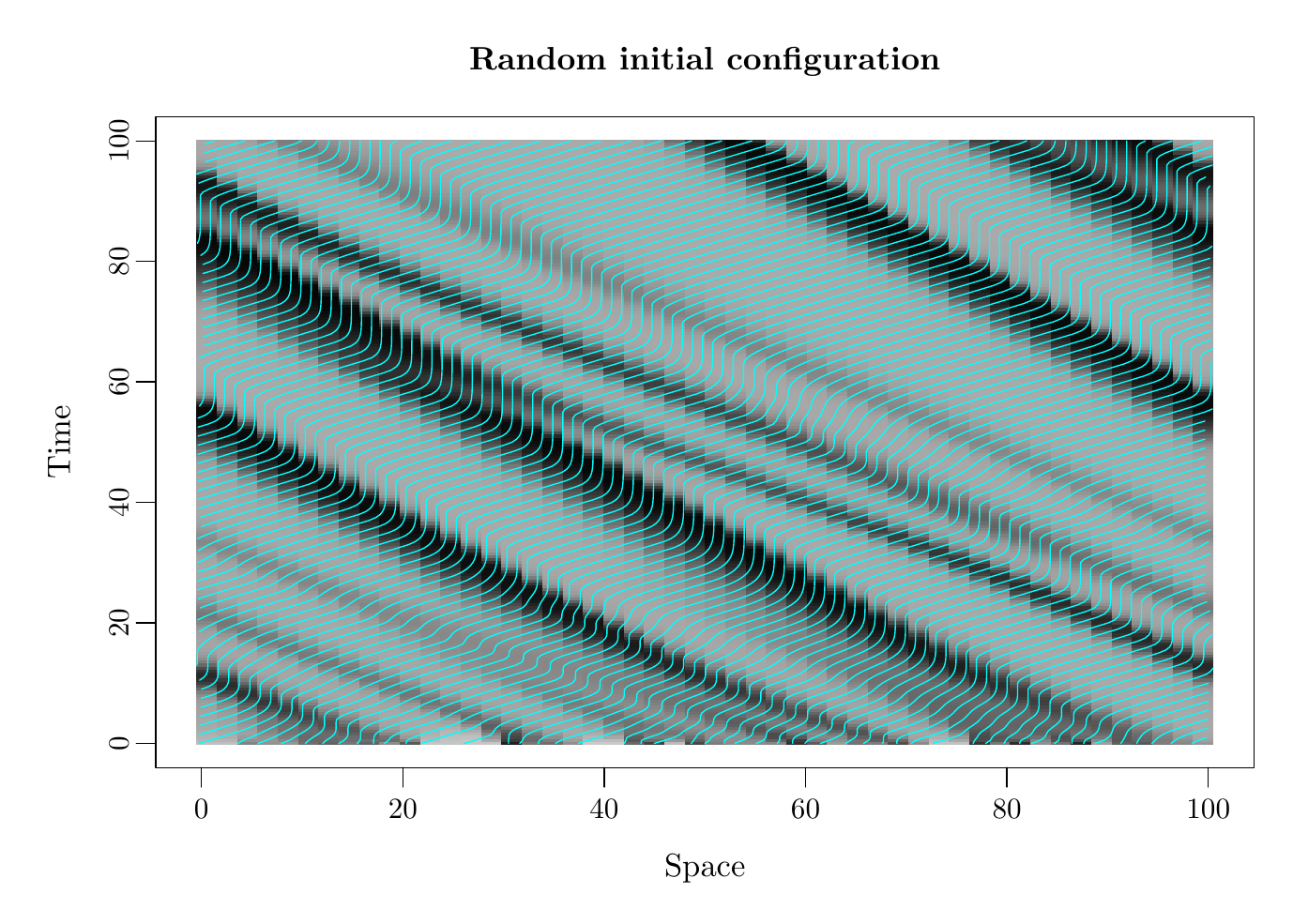}\vspace{-4mm}
\caption{The trajectories of the microscopic model 
(cyan curves) and the time series for the density by cell for the
discrete macroscopic model (gray levels) for random initial conditions.}\label{spti1}
\vspace{0mm}\ec\efig

\bfig\bc
\includegraphics[width=.81\textwidth]{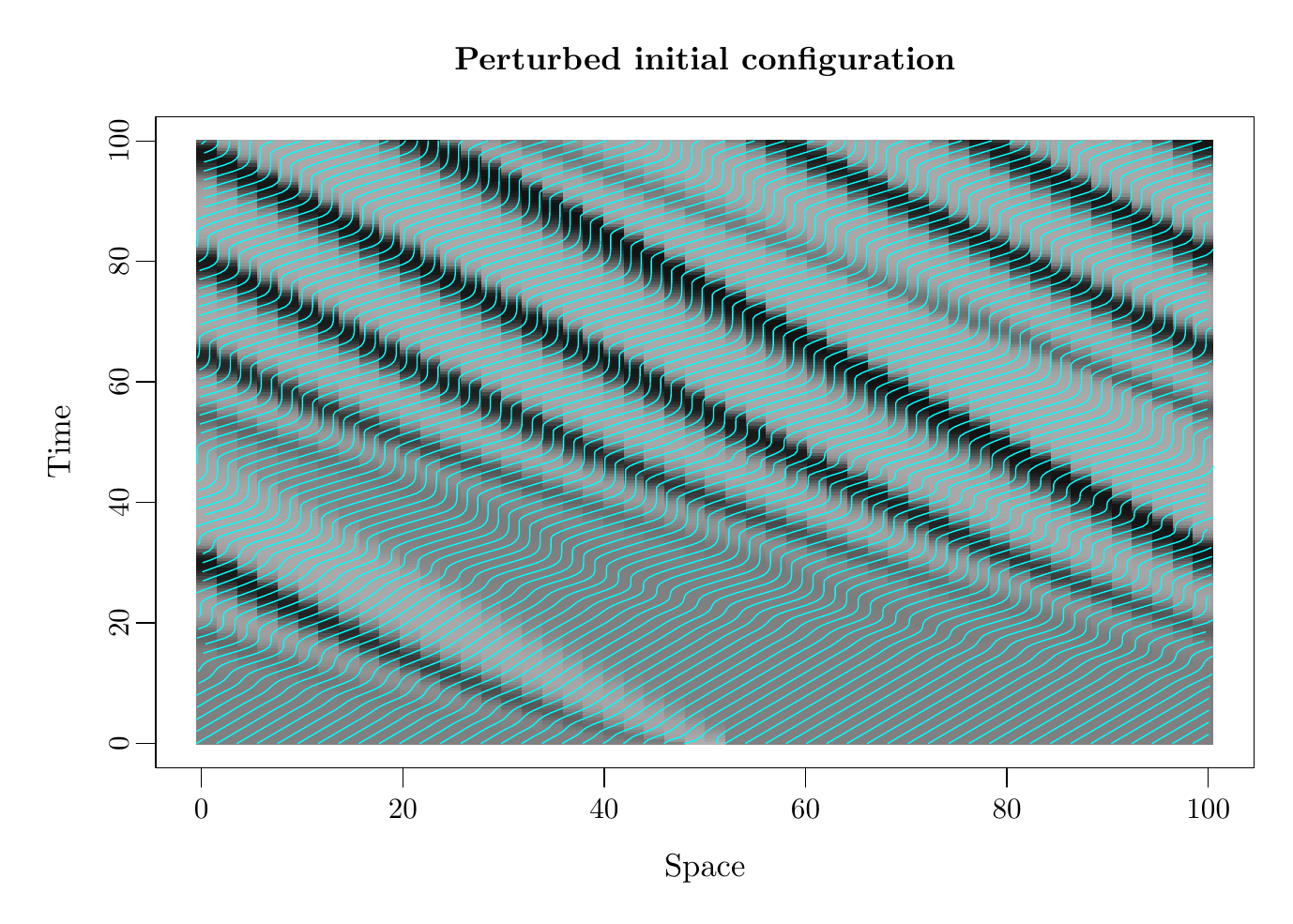}\vspace{-4mm}
\caption{The trajectories of the microscopic model 
(cyan curves) and the time series for the density by cell for the
discrete macroscopic model (gray levels) for perturbed initial conditions.}\label{spti2}
\ec\efig

\subsection{Fundamental diagram}

The fundamental diagram is the plot of the flow or the mean speed as a function of the density. 
It generally refers to spatial performances \cite{Edie1963}, 
that have to be distinguished from temporal ones \cite{Wardrop1952}. 
Here to deal with spatial performances, we measured the spatial speed and the density and express the flow 
as the product of the density by the speed. 
The density for the microscopic model is the inverse of the spacing (see Eq.~(\ref{dr})) 
while the speed in the macroscopic model is (see Eq.~(\ref{Vtilde}))
\be
\tilde V(\rho_i,\rho_{i+1})=V\left(\frac \rho{1-\tau\rho (V(\rho_{i+1})-V(\rho_i))}\right).
\ee
The sequences obtained for the perturbed initial conditions (see Fig.~\ref{spti2}) 
are presented in Fig.~\ref{evol}. 
The performances are \textit{instantaneous} ones in the sense that they correspond to instantaneous measurements for a vehicle (microscopic model) 
and a cell (macroscopic) in the system. 
The variability in such diagram is larger than the one of the \textit{aggregated} fundamental diagram plotted in Fig.~\ref{fd-data} 
where the performances were averaged over time intervals.

\bfig\vspace{0mm}\bc
\hspace{-.1\textwidth}\includegraphics[width=.84\textwidth]{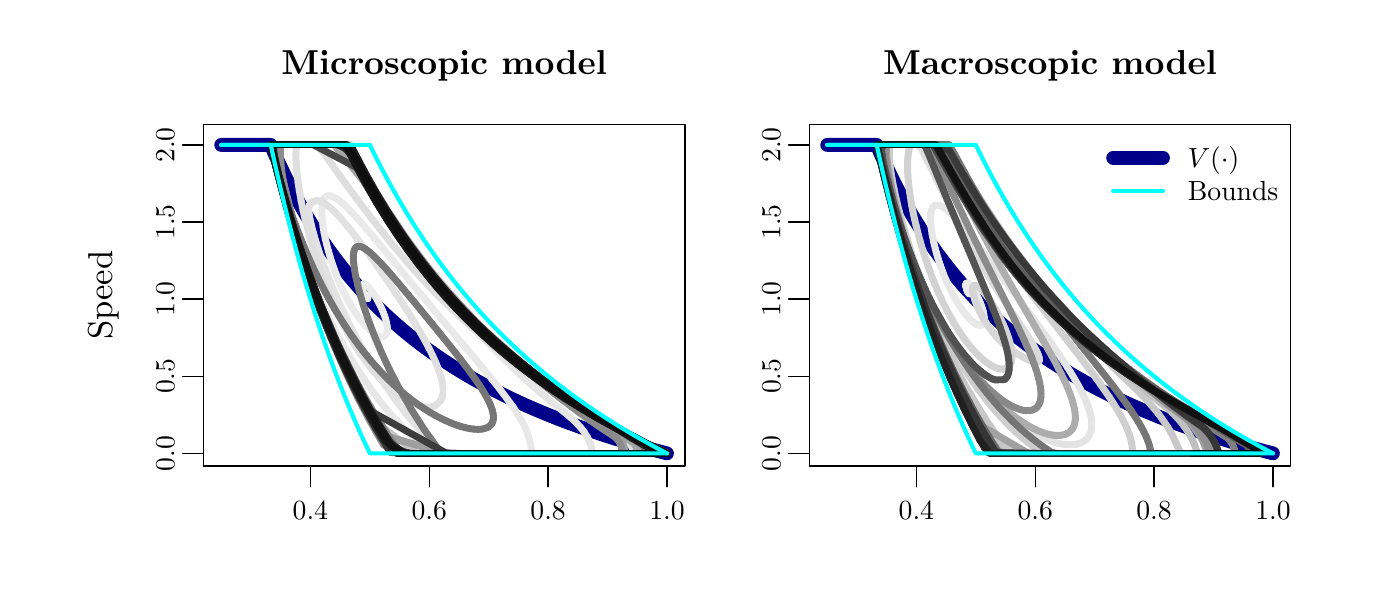}\\[-9mm]
\hspace{-.1\textwidth}\includegraphics[width=.84\textwidth]{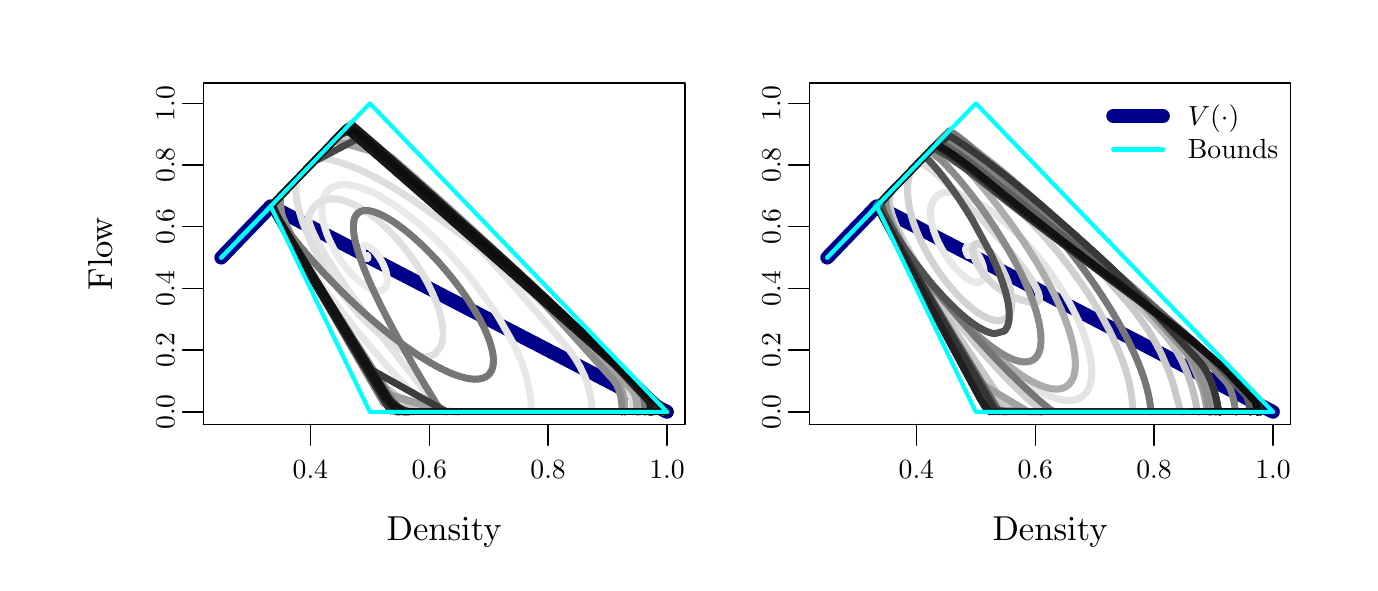}\\[-90mm]
\hspace{.85\textwidth}\includegraphics[width=.125\textwidth]{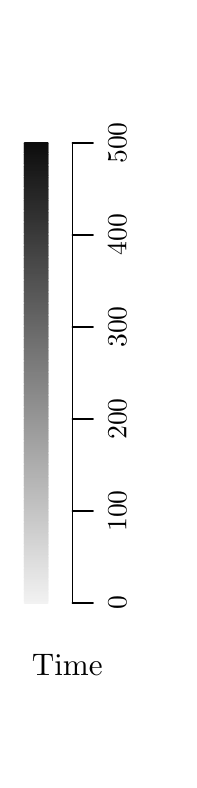}\\[10mm]
\caption{Sequence of speed and flow / density relation for the perturbed initial conditions (see Fig.~\ref{spti2}).
Left, for one vehicle (microscopic model) and right, for one cell (macroscopic model).}\label{evol}
\vspace{0mm}\ec\efig

Both microscopic and macroscopic systems converge to limit-cycles with self-sustained stop-and-go waves resulting in hysteresis curves in the 
microscopic fundamental diagram. 
Such phenomenon generate scattering of the fundamental diagram for which some bounds can be calculated 
\cite{Zhang2002,Colombo2003,Goatin2006,Colombo2010,Seibold2013,Fan2014}. 
The bounds $V^+$ and $V^-$ for the fundamental diagrams can here intuitively been determined from the microscopic model. 
The upper bound $V^+$ corresponds to the sequence of a vehicle moving at maximal speed $V_0$ behind a stopped vehicle:
\be
V^+(\rho)=\tilde V(\rho,1/\ell)=V\left(\frac \rho{1+\tau\rho V(\rho)}\right).
\label{V+}
\ee 
Due to the reaction time, the distance tends to be smaller and the fundamental diagram is `over-estimated'. 
Oppositely, the lower bound $V^-$ corresponds to the sequence of a stopped vehicle following a predecessor moving at the maximal speed $V_0$:
\be
V^-(\rho)=\tilde V(\rho_i,0)=V\left(\frac \rho{1-\tau\rho (V_0-V(\rho))}\right).
\label{V-}
\ee 
Here the reaction time induces a delay in the acceleration and an under-estimation of the fundamental diagram.

As in \cite{Seibold2013,Fan2014}, the bounds Eqs.~(\ref{V+}) and (\ref{V-}) obtained with the macroscopic model are compared 
to real instantaneous pedestrians and road traffic data in Figs.~\ref{fd-ped} and \ref{fd-auto}. 
The pedestrians data comes from a laboratory experiment with participants in a ring geometry \cite{peddata}. 
Several experiments have been carried out with different density levels. 
The road traffic data are real measurement of trajectories on an American highway \cite{ngsim}. 
The speed, density and the flow are measured as previously 
(i.e.~the density is the inverse of the spacing while the flow is the product of the density by the speed). 
A triangular fundamental diagram with 3 parameters $V(\rho)=\min\big\{V_0,\frac1T(1/\rho-1)\big\}$ is used again. 
The parameters are the ones of an estimation by least squares 
for the pedestrians $V_0=0.9$~m/s, $\ell=0.3$~m and $T=\tau=1$~s, see Fig.~\ref{fd-ped}, 
while $V_0=15$~m/s, $\ell=5$~m and $T=\tau=2$~s for the vehicles, see Fig.~\ref{fd-auto}. 
The bounds present a reasonable agreement with the data, even if no clustering of measurements are observed around them.

\bfig\bc
\includegraphics[width=.9\textwidth]{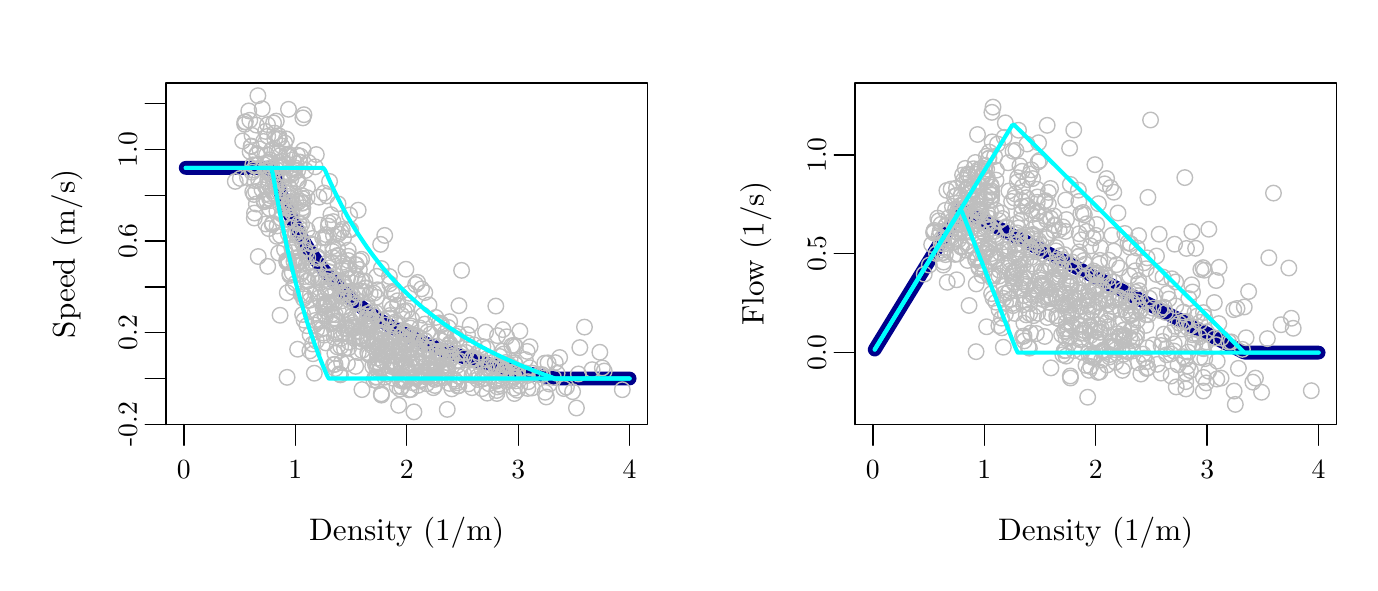}\vspace{-2mm}
\caption{Instantaneous speed/density and flow/density measurements 
for real pedestrian flows \cite{peddata} and the bounds Eqs.~(\ref{V+}) and (\ref{V-}) 
for $V_0=0.9$~m/s, $\ell=0.3$~m and $T=\tau=1$~s.}\label{fd-ped}
\ec\efig

\bfig\bc\vspace{0mm}
\includegraphics[width=.9\textwidth]{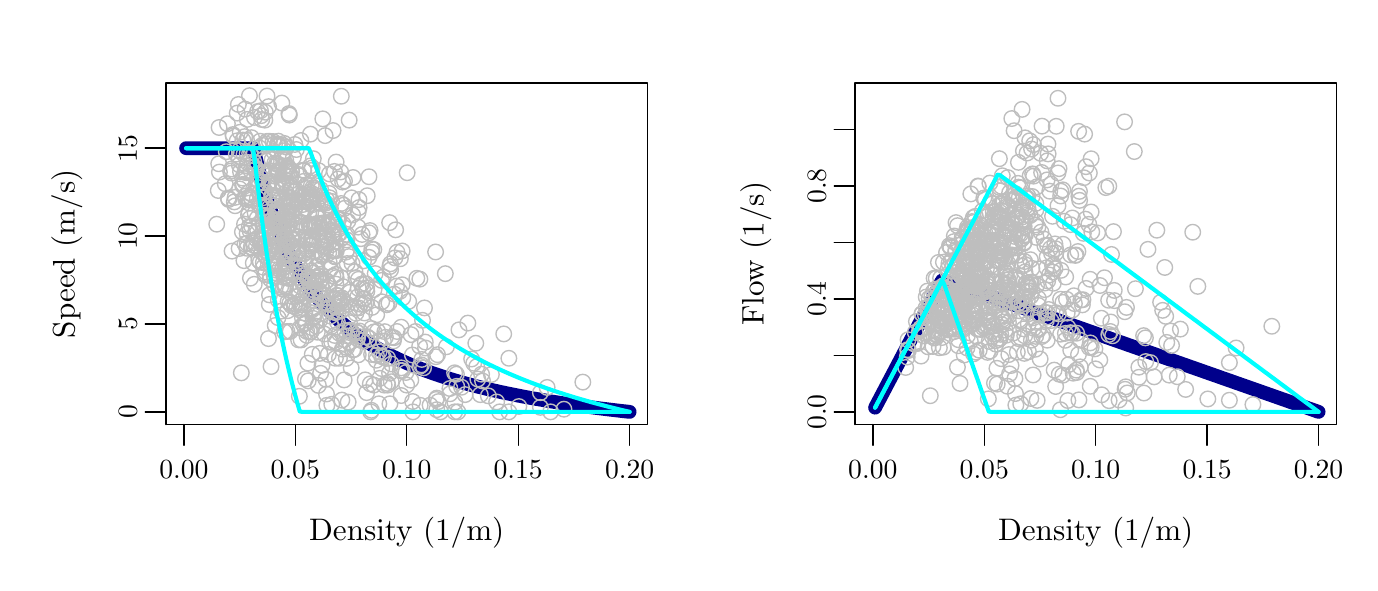}\vspace{-2mm}
\caption{Individual speed/density and flow/density measurements 
for real road traffic flows \cite{ngsim} and the bounds Eqs.~(\ref{V+}) and (\ref{V-}) 
for $V_0=15$~m/s, $\ell=5$~m and $T=\tau=2$~s.}\label{fd-auto}\vspace{0mm}
\ec\efig

\section{Conclusion}\label{ccl}

Starting from a speed following model, we derive a first order convection-diffusion continuum traffic flow model 
that we discretised using Godunov and Euler schemes. 
Simulation results shown that discrete macroscopic models can recapture the dynamics of the microscopic model, 
if specific values for the space discretization are chosen.  
More precisely, the linear stability conditions of the homogeneous solutions for the macroscopic models
match the ones of the microscopic model for specific values of the space discretization and sufficiently small time steps. 

For unstable conditions, i.e.~for large reaction times, the dynamics obtained describe self-sustained stop-and-go waves, 
with hysteresis cycles and a large scattering of the fundamental flow/density diagram. 
Such characteristics are observed in real data \cite{Treiterer1974,Kerner1997,Kerner2000,Chowdhury2000,ZhangJ2014} 
as well as for second order models \cite{Zhang2002,Colombo2003,Goatin2006,Berthelin2008,Flynn2009,Seibold2013,Colombo2010,Fan2014}. 
Here it is achieved with first order models ensuring by construction that the models are physical and `collision-free' 
(i.e. bounded and positive speed as well as density).
Further investigations are necessary to understand the impact of the shape of the optimal velocity function 
on the characteristics of the waves. 

The macroscopic model corresponding to the follow-the-leader model is a first order elliptic convection-diffusion equation, for which 
the convection part is calibrated by the optimal velocity function (i.e.~the fundamental diagram), while the diffusion 
is proportional to the reaction time parameter. 
More precisely the diffusion is negative in deceleration phases where the density get higher, and it is positive 
in acceleration phases where the density decreases. 
Such mechanism seems to be responsible for the appearance of oscillations and self-sustained non-linear stop-and-go waves in the system. 
This observation remains to be confirmed rigorously, yet it could give us a way to explain the wave formations. 

\bibliographystyle{plain}
\bibliography{bibli,completeBibTex}

\newpage
\tableofcontents

\end{document}